\documentstyle[11pt,epsfig,mst-stylefile,amssymb,amsmath,harvard]{article}
\newcommand{\bbR}{{\Bbb R}}
\newcommand{\bbN}{{\Bbb N}}
\newcommand{\bbC}{{\Bbb C}}

\newcommand {\ind} {\overset{\mathcal{L}}{=}}
\newcommand {\tr} {{\rm tr}}
\newcommand {\eig} {{\rm eig}}

\newcommand{\norm}[1]{\left\|#1\right\|}

\renewcommand{\cite}{\citeyear}

\begin{document}

\title{On integral representations of operator\\ fractional Brownian fields
\thanks{The second author was supported in part by the Louisiana Board of Regents award LEQSF(2008-11)-RD-A-23.
The third author was supported in part by the NSA grant H98230-13-1-0220.}
 }

\author{Changryong Baek \\ Sungkyunkwan University \and Gustavo Didier \\ Tulane University  \and Vladas Pipiras \\ University of North Carolina}

\bibliographystyle{agsm}

\maketitle

\begin{abstract}
Operator fractional Brownian fields (OFBFs) are Gaussian, stationary-increment vector random fields that satisfy the operator self-similarity relation $\{X(c^{E}t)\}_{t \in \bbR^m} \stackrel{{\mathcal L}}= \{c^{H}X(t)\}_{t \in \bbR^m}$. We establish a general harmonizable representation (Fourier domain stochastic integral) for OFBFs. Under additional assumptions, we also show how the harmonizable representation can be reexpressed as a moving average stochastic integral, thus answering an open problem described in Bierm\'{e} et al.\ \cite{bierme:meerschaert:scheffler:2007}.
\end{abstract}

%---------------------------------INTRODUCTION--------------------------------------

\section{Introduction}
\label{s:intro}

An {\it operator fractional Brownian field} (OFBF, in short) is a vector-valued random field $X = \{X(t)\}_{t\in \bbR^m} \in\bbR^n$ satisfying the following three properties: $(i)$ it is Gaussian with mean zero; $(ii)$ it has stationary increments, that is, for any $h\in\bbR^m$, $\{X(t+h) - X(h)\}_{t\in\bbR^m} \stackrel{{\mathcal L}}= \{X(t) - X(0)\}_{t\in\bbR^m}$, where $\stackrel{{\mathcal L}}=$ denotes the equality of finite-dimensional distributions; $(iii)$ it is operator self-similar in the sense that there are real square matrices $E$, $H$ of appropriate dimensions so that
\begin{equation}\label{e:operator_scaling}
\{X(c^{E}t)\}_{t \in \bbR^m} \stackrel{{\mathcal L}}= \{c^{H}X(t)\}_{t \in \bbR^m}, \quad c > 0.
\end{equation}
For any square matrix $M$, the term $c^M$ is defined in the usual way as $\sum^{\infty}_{k=0} \frac{(\log c)^k M^k }{k!}$. In this paper, OFBFs are assumed to be proper in the sense that, for each $t\in \bbR^m$, the distribution of $X(t)$ is not contained in a proper subset of $\bbR^n$.

When both the domain and range are unidimensional ($m=1$, $n=1$ and $E=1$), an OFBF is simply the celebrated fractional Brownian motion (FBM). When the \textit{domain} is multidimensional ($m\geq 2$, $n=1$), OFBFs are referred to as operator scaling Brownian fields. The latter and, more generally, operator scaling random fields were studied in Bonami and Estrade \cite{bonami:estrade:2003}, Bierm\'{e}, Meerschaert and Scheffler \cite{bierme:meerschaert:scheffler:2007}, Bierm\'{e} and Lacaux \cite{bierme:lacaux:2009}, Bierm\'{e}, Benhamou and Richard \cite{bierme:benhamou:richard:2009}, Bierm\'{e}, Estrade and Kaj \cite{bierme:estrade:kaj:2010}, Clausel \cite{clausel:2010}, Clausel and V\'{e}del \cite{clausel:vedel:2013}. When the \textit{range} is multidimensional ($m=1$, $n\geq 2$ and $E=1$), OFBFs are called operator fractional Brownian motions (OFBMs). The latter and, more generally, operator self-similar stochastic processes were studied in Laha and Rohatgi \cite{laha:rohatgi:1981}, Hudson and Mason \cite{hudson:mason:1982}, Maejima and Mason \cite{maejima:1994}, Cohen et al.\ \cite{cohen:meerchaert:rosinski:2010}, Didier and Pipiras \cite{didier:pipiras:2011,didier:pipiras:2012}.

This paper constitutes a step towards the synthesis of the aforementioned domain- and range-based streams of literature. So far, few works have accounted for both sources of multidimensionality, such as Li and Xiao \cite{li:xiao:2011}. Our contribution is two-fold, namely, constructing widely encompassing harmonizable (Fourier domain) integral representations of OFBFs, as well as establishing their relation to moving average (time domain) representations.

The starting point is to build harmonizable representations (Theorem \ref{t:ofbf_harmonizable_repres}). Such representations for various subclasses of OFBFs have been considered in the literature and proved useful, for instance, in the study of sample path properties. The novelty behind our results lies in necessity, namely, the way the properties of OFBFs tangibly shape the spectral measure. This leads to harmonizable representations that characterize the entire class of OFBFs. When $m=1$, the analogous result can be found in Didier and Pipiras \cite{didier:pipiras:2011}; however, the characterization provided in this work is new even in the case of scalar fields, i.e., $n=1$ and $m\geq 2$ (see Clausel and Vedel \cite{clausel:vedel:2011} for a pseudo-norm perspective). In Didier et al.\ \cite{didier:meerschaert:pipiras:2014}, the harmonizable representations developed in this paper are the main analytical tool for the comparison, synthesis and scrutiny of the domain and range symmetries of general OFBFs.

Recent studies have begun to reveal the ways symmetry, or the lack thereof, make multidimensional fractional systems different from their unidimensional counterparts (see Didier and Pipiras \cite{didier:pipiras:2011} on time-reversibility). In the scalar fields case ($m \geq 2$, $n=1$), Bierm\'{e} et al.\ \cite{bierme:meerschaert:scheffler:2007}, p.325, start off with two isotropic OFBFs $X_{\varphi}$, $X_{\psi}$ with the same scaling parameters $E = I$ and $H \in \bbR$, but defined, respectively, from moving average and harmonizable representations generated by the (operator-homogeneous; see \eqref{e:H-E-left-homogeneous}) kernel filters $\varphi$ and $\psi$. It can then be shown that
\begin{equation}\label{e:cov_function_under_isotropy}
E X_{\Lambda}(s)X_{\Lambda}(t) = \frac{c_{\Lambda}}{2} \{\|s\|^{2H} + \|t\|^{2H} - \| s-t \|^{2H}  \}, \quad \Lambda = \varphi, \psi,
\end{equation}
whence $X_{\varphi}$ and $X_{\psi}$ are equal in law up to a constant. However, under a general $E$ and, thus, anisotropy, the terms $\|\cdot\|^{2H}$ and $c_{\Lambda}$ in \eqref{e:cov_function_under_isotropy} are replaced by $\tau(\cdot)^{2H}$ and $\sigma^2_{\Lambda}(l(\cdot))$, respectively, for radial and spherical functions $\tau$ and $l$ that depend on $E$ (see \eqref{e:x=tau(x)E*l(x)} below). Consequently, it is difficult to compare $\sigma^2_{\varphi}(l(\cdot))$ and $\sigma^2_{\psi}(l(\cdot))$. Bierm\'{e} et al \cite{bierme:meerschaert:scheffler:2007} then go on to describe the question that has been open to date: under what relationship between the moving average and harmonizable filters $\varphi$, $\psi$ are the respective integral representations equivalent?

A natural way to link harmonizable and moving average representations is by taking the Fourier transform of one of the kernel functions. In the case of OFBMs ($m=1$, $n\geq 2$), this approach was applied in Didier and Pipiras \cite{didier:pipiras:2011}. The case of OFBFs ($m\geq 2$), on the other hand, is more challenging. In particular, it can be shown that the natural integration procedure for homogeneous functions only holds when, indeed, $m=1$ (see Remark \ref{re:tech diff} below). In this paper, we show how the moving average representation for OFBFs can be obtained through the Fourier transform under additional mild assumptions on the harmonizable representation (Theorem \ref{t:int-rep-time}). We thus provide an answer to the open problem described in Bierm\'{e} et al.\ \cite{bierme:meerschaert:scheffler:2007}.

The paper is organized as follows. In Section \ref{s:preliminaries}, the notation and necessary background from the theory of operator self-similar vector random fields are laid out. In Section \ref{s:integral_repres_OFBF}, the harmonizable and moving average representations for OFBFs are established.

\section{Basic definitions and notation}
\label{s:preliminaries}

In this section, we lay out the notation and conceptual framework used in the paper.

$M(n)$ and $M(n,\bbC)$ denote, respectively, the spaces of $n \times n$ matrices with real or complex entries, whereas the space of $n \times m$ matrices with real entries is denoted by $M(n, m, \bbR)$. $A^*$ stands for the Hermitian transpose of $A \in M(n, \bbC)$, whereas $\| \cdot \|$ indicates the Euclidean norm.

The Jordan spectra (characteristic roots) of the matrices $E$ and $H$ in (\ref{e:operator_scaling}) are written
\begin{equation}\label{e:hk}
\textnormal{eig}(E) = \{e_1, e_2, \hdots, e_{m}\}, \quad \textnormal{eig}(H) = \{h_1, h_2, \hdots, h_n\}.
\end{equation}
Specific assumptions on the range of eigenvalues will be made in propositions. An OFBF with exponents $E$ and $H$ is written $B_{E,H} = \left\{ B_{E,H} (t) \right\}_{t \in \bbR^m}$.

We will make wide use of changes-of-variables into non-Euclidean polar coordinates, as discussed in Meerschaert and Scheffler \cite{meerschaert:scheffler:2001} and Bierm\'{e} et al.\ \cite{bierme:meerschaert:scheffler:2007}. Suppose the real parts of the eigenvalues of $E \in M(m)$ are positive. Then, there exists a norm $\| \cdot \|_0$ on $\bbR^m$ for which
\begin{equation}\label{e:Phi_change-of-variables}
\Psi:(0,\infty) \times S_0 \rightarrow \bbR^m \backslash\{0\},
\quad \Psi(r,\theta):=r^{E}\theta,
\end{equation}
is a homeomorphism, where $S_0 = \{x \in \bbR^m: \norm{x}_{0}=1\}$. One can then uniquely write the {\it polar coordinates} representation
\begin{equation}\label{e:x=tau(x)E*l(x)}
\bbR^m \ni x = \tau_E(x)^E l_E(x),
\end{equation} where the functions $\tau_E(x)$, $l_E(x)$ are called the {\it radial} and {\it directional} parts, respectively. One such norm $\| \cdot \|_0$ may be calculated explicitly by means of the expression
\begin{equation}\label{e:E0_norm}
\norm{x}_{0} = \int^{1}_{0} \norm{t^{E}x}_{*}\frac{dt}{t},
\end{equation}
where $\| \cdot \|_{*}$ is any norm in $\bbR^m$.
The uniqueness of the representation (\ref{e:x=tau(x)E*l(x)}) yields
\begin{equation}\label{e:tau_E}
\tau_E(c^E x) = c \tau_E(x), \quad l_E(c^E x) = l_E(x).
\end{equation}
Based on $\| \cdot \|_{0}$, a {\it change-of-variables formula} into polar coordinates is established in Bierm\'{e} et al.\ \cite{bierme:meerschaert:scheffler:2007}, Proposition 2.3: there exists a unique finite Radon measure $\sigma( d \theta)$ such that, for any scalar-valued $h \in L^1(\bbR^m)$,
\begin{equation}\label{e:polar_coordinates}
\int_{\bbR^m} h(x) dx = \int^{\infty}_0 \int_{S_0} h(r^E \theta) r^{\tr(E)-1} \sigma(d \theta) dr.
\end{equation}

We shall use bounds on the radial part $\tau_E(x)$. Let $e_{\min} = \min \Re( \eig(E))$, $e_{\max} = \max \Re( \eig(E))$. In \eqref{e:x=tau(x)E*l(x)}, the radial component is related to the norm $\|\cdot\|_0$ in the sense that, for all $\delta > 0$, there are constants $C_j$, $j=1,\hdots, 4$, such that
\begin{equation}\label{e:tau(x)_vs_|x|_small_|x|}
C_1 \|x\|^{1/(e_{\min}+\delta)}_0 \leq \tau_E(x) \leq C_2 \|x\|^{1/(e_{\max}-\delta)}_0, \quad \|x\|_0 \leq 1, \quad \tau_E(x)\leq 1,
\end{equation}
\begin{equation}\label{e:tau(x)_vs_|x|_large_|x|}
C_3 \|x\|^{1/(e_{\max}-\delta)}_0 \leq \tau_E(x) \leq C_4 \|x\|^{1/(e_{\min}+\delta)}_0, \quad \|x\|_0 > 1, \quad \tau_E(x)> 1
\end{equation}
(Bierm\'{e} et al.\ \cite{bierme:meerschaert:scheffler:2007}, p.\ 314).
Note that, since all norms on $\bbR^m$ are equivalent, $\| \cdot \|_0$ can be replaced by any other norm in (\ref{e:tau(x)_vs_|x|_small_|x|})--(\ref{e:tau(x)_vs_|x|_large_|x|}).

When $n=1$, Meerschaert and Scheffler \cite{meerschaert:scheffler:2001} use $E$--homogeneous functions, namely, functions $\varphi: \bbR^m \rightarrow \bbC$ which satisfy the scaling relation $\varphi(c^Ex) = c\varphi(x)$, $c > 0$, $x \in \bbR^m\backslash\{0\}$. However, when the functions $\varphi$ are matrix-valued, like those often encountered in the framework of this paper, more than one notion of operator-homogeneity is of interest. This idea is laid out in detail in the next definition.

\begin{definition}
Let $E\in M(m)$, $H\in M(n)$, and let $\varphi: \bbR^m \rightarrow M(n,\bbC)$. The matrix-valued function $\varphi$ is called $(E,H)$--{\it left-homogeneous} if for any $c > 0$ and $x \in \bbR^m \backslash\{0\}$,
\begin{equation} \label{e:H-E-left-homogeneous}
\varphi(c^E x) = c^{H} \varphi(x),
\end{equation}
whereas it is called $(E,H)$--{\it homogeneous} if
\begin{equation}\label{e:H-E-homogeneous}
\varphi(c^E x) = c^{H/2} \varphi(x) c^{H^*/2}.
\end{equation}
\end{definition}
Although one can straightforwardly conceive the notion of right-homogeneity, it will not be used in our analysis and will thus be omitted.

Polar coordinate representations provide a convenient means of expressing $(E,H)$--left-homogeneous or $(E,H)$--homogeneous functions. In the case of the latter, note that taking $c = \tau_E(x)^{-1}$ in (\ref{e:H-E-homogeneous}) yields $\varphi(\tau_E(x)^{-E} x) = \tau_E(x)^{-H/2}\varphi(x)\tau_E(x)^{-H^*/2}$, and hence
\begin{equation}\label{e:multidim-mfbm-homog-f-polar}
\varphi(x) = \tau_E(x)^{H/2} \varphi(l_E(x))\tau_E(x)^{H^*/2}.
\end{equation}
In particular, $\varphi$ is determined by its values on the sphere $S_0$. Conversely, by using (\ref{e:tau_E}), the right-hand side of (\ref{e:multidim-mfbm-homog-f-polar}) always defines an $(E,H)$--homogeneous function. In this paper, the concept of left-homogeneity is associated with filters that go into integral representations, whereas the concept of homogeneity applies to spectral densities, or ``squares". However, note that an $(E,H)$--homogeneous $\varphi$ does not necessarily yield a positive semidefinite matrix at a given $x$, since $\varphi(l_E(x))$ may have complex eigenvalues.

The proof of the next technical lemma illustrates the use of the polar coordinate representation (\ref{e:x=tau(x)E*l(x)}). The lemma is used in the proof of Theorem \ref{t:ofbf_harmonizable_repres} below.

\begin{lemma}\label{l:multidim-mfbm-homog-f-ae}
Suppose $E \in M(m)$ has eigenvalues with positive real parts. If $h: \bbR^m \to M(n, \bbC)$ is such that, for all $c > 0$,
\begin{equation}\label{e:multidim-mfbm-homog-f-ae}
    h(c^Ex) = c^{H/2} h(x) c^{H^*/2}\quad dx\mbox{-a.e.},
\end{equation}
then there exists an $(E,H)$--homogeneous function $\varphi$ such that $\varphi(x) = h(x)$ $dx$-a.e.
\end{lemma}
\begin{proof}
Let $x \in \bbR^m \backslash \{0\}$ be such that the relation \eqref{e:multidim-mfbm-homog-f-ae} holds. Then,
$$
\int_0^
\infty 1_{\left\{h(c^{-E}x) \neq c^{-H/2}h(x) c^{-H^*/2} \right\}} (c, x)   dc = 0.
$$
Equivalently, by making the change of variables $y = c \tau_E(x)^{-1}$,
$$
0 = \tau_E(x) \int_0^\infty 1_{ \left\{ h(y^{-E}l_E(x)) \neq y^{-H/2}\tau_E(x)^{-H/2}h(x) \tau_E(x)^{-H^*/2}y^{-H^*/2} \right\}} (y, x) dy .
$$
As a consequence, since \eqref{e:multidim-mfbm-homog-f-ae} holds $dx\textnormal{-a.e.}$,
$$
\int_{\bbR^n} \int_0^\infty 1_{ \left\{ h(c^{-E}l_E(x)) \neq c^{-H/2}\tau_E(x)^{-H/2}h(x) \tau_E(x)^{-H^*/2}c^{-H^*/2} \right\} } (c, x) dcdx = 0.
$$
By Fubini's theorem,
$$
h(x) = \tau_E(x)^{H/2}c^{H/2}h(c^{-E} l_E(x)) c^{H^*/2}\tau_E(x)^{H^*/2} \quad dxdc\textnormal{-a.e.}
$$
Therefore, there exists $c_0 > 0$ such that
\begin{equation}\label{e:h(x)=scaling_relation_with_fixed_c0}
h(x) = \tau_E(x)^{H/2}c^{H/2}_0 h(c^{-E}_0 l_E(x))c^{H^*/2}_0 \tau_E(x)^{H^*/2} \quad dx\textnormal{-a.e.}
\end{equation}
In view of \eqref{e:h(x)=scaling_relation_with_fixed_c0}, define $\varphi(x)= \tau_E(x)^{H/2}c^{H/2}_0 h(c^{-E}_0 l_E(x))c^{H^*/2}_0 \tau_E(x)^{H^*/2}$, $x \in \bbR^m \backslash\{0\}$, which is an $(E,H)$--homogeneous function as discussed following \eqref{e:multidim-mfbm-homog-f-polar}. $\Box$\\
\end{proof}

\begin{remark} A result analogous to Lemma \ref{l:multidim-mfbm-homog-f-ae} can be similarly established for left-homogeneous functions. That is, if $h$ satisfies $h(c^E x) = c^H h(x)$ $dx\textnormal{-a.e.}$ for any fixed $c > 0$, then there is an $(E,H)$--left-homogeneous function $\varphi$ so that $h(x) = \varphi(x)$ $dx\textnormal{-a.e.}$

\end{remark}

%%%%%%%%%%%%%%%%%%%%%%%%%%%%%%%%%%%%%%%%%%%%%%%%%%%%%%%%%%%%%%%%%%%%%%%%%%%%%%%%%%%%%%%%%%%%%%%%%%%%%%%%%%%
% old proof

%\noindent {\sc Proof:} Fubini's theorem implies that
%$$
%h(c^Ex) = c^{H/2} h(x)c^{H^*/2},\quad dcdx\mbox{-a.e.}
%$$
%On the other hand, by the change-of-variables $(c\mapsto c\tau_E(x)^{-1}, x\mapsto x)$,
%$$
%h(x) = c^{-H/2} \tau_E(x)^{H/2}  h(c^E l_E(x)) \tau_E(x)^{H/2} c^{-H^*/2}\quad dcdx\mbox{-a.e.}
%$$
%Now, there is $c_0$ such that the last relation holds a.e.\ $dx$, that is,
%$$
%h(x) = \tau_E(x)^{H/2}  c^{-H/2}_0 h(c^E_0 l_E(x)) c^{-H^*/2}_0 \tau_E(x)^{H/2} \quad dx\mbox{-a.e.}.
%$$
%As noted following expression (\ref{e:multidim-mfbm-homog-f-polar}), the right-hand side in the last relation defines an $(E,H)$--homogeneous function.
%\quad \quad $\Box$
%%%%%%%%%%%%%%%%%%%%%%%%%%%%%%%%%%%%%%%%%%%%%%%%%%%%%%%%%%%%%%%%%%%%%%%%%%%%%%%%%%%%%%%%%%%%%%%%%%%%%%%%%%%

\section{Integral representations of OFBFs}
\label{s:integral_repres_OFBF}

\subsection{Harmonizable representations}

In the next result, we establish harmonizable representations of OFBFs. The representations draw upon random measures having the following properties. A $\bbC^n$--valued Gaussian random measure $\widetilde{Y}$ on $\bbR^m$ will be called Hermitian with matrix-valued control measure $F$ if it satisfies $E \widetilde{Y}(dx) \widetilde{Y}(dx)^* = F(dx)$ and $\widetilde{Y}(-dx) = \overline{\widetilde{Y}(dx)}$. The associated distribution function $F$ then takes values in the cone of Hermitian positive semidefinite matrices.

\begin{theorem}\label{t:ofbf_harmonizable_repres}
Let $B_{E,H} = \{B_{E,H}(t)\}_{t\in\bbR^m}$ be an OFBF with matrix exponents $E$ and $H$ satisfying
\begin{equation}\label{e:exist_harmon_representation_conditions_eigens}
0 < \min \Re(\textnormal{eig}(H)) \leq \max  \Re(\textnormal{eig}(H)) < \min \Re(\textnormal{eig}(E^*)).
\end{equation}
\begin{itemize}
\item [$(i)$] Then, $B_{E,H}$ has a harmonizable representation
\begin{equation}\label{e:int-rep-spectral-F}
\{B_{E,H}(t)\}_{t\in\bbR^m} \stackrel{{\mathcal L}}{=} \Big\{\int_{\bbR^m} (e^{i\langle x,t\rangle} - 1)  \widetilde{B}_F(dx) \Big\}_{t\in\bbR^m},
\end{equation}
where $\widetilde{B}_F(dx)$ is an Hermitian Gaussian random measure on $\bbR^n$ with control measure $F(dx)$ such that $F$ is $(-E^*, 2H)$--homogeneous, that is,
\begin{equation}\label{e:int-rep-spectral-F-homog-meas}
F(d(c^{-E^*}x))= c^{H}F(dx)c^{H^*}, \quad c>0.
\end{equation}
\item [$(ii)$] Let $H_E = H + \textnormal{tr}(E^*)I/2$, and assume that the spectral measure $F(dx)$ is absolutely continuous with density $f(x)$. Then, there exists an $(E,2 H_E)$--homogeneous function $\psi$ such that $\psi(x) = f(x)$ $dx\textnormal{-a.e.}$ and
\begin{eqnarray}\label{e:int-rep-spectral-f-1}
  \{B_{E,H}(t)\}_{t\in\bbR^m}  \stackrel{{\mathcal L}}{=} \Big\{\int_{\bbR^m} (e^{i\langle x,t\rangle} - 1)  g(x) \widetilde{B}(dx) \Big\}_{t\in\bbR^m},
\end{eqnarray}
where $g(x) = \psi^{1/2}(x)$ a.e., and $\widetilde{B}(dx)$ is an Hermitian Gaussian random measure on $\bbR^m$ with Lebesgue control measure. Moreover, $g$ is $(E^*,-H_E)$--left-homogeneous.
\end{itemize}
\end{theorem}
\begin{proof}
By the general harmonizable representation for vector random fields with stationary increments in Yaglom \cite{yaglom:1957}, Theorem 7, the OFBF $B_{E,H}$ can be represented as
\begin{equation}\label{e:yaglom_s_representation}
\{ B_{E,H}(t) \}_{t \in \bbR^m}\stackrel{{\mathcal L}}= \Big\{\int_{\bbR^m} (e^{i \langle t ,x \rangle}-1) \widetilde{B}_F(dx) + X t + Y \Big\}_{t \in \bbR^m},
\end{equation}
for some Hermitian Gaussian random measure $\widetilde{B}_F(dx)$ with a (symmetric) control measure $F(dx)$. The terms $Y$ and $X$ are, respectively, a Gaussian random vector in $\bbR^n$ and a Gaussian random matrix in $M(n,m,\bbR)$ with $ \mu : = EX $ and $\Sigma_X := \textnormal{Var} (\textnormal{vec}(X))$. The random components $\widetilde{B}_F(dx)$, $X$, $Y$ are independent. Since $\min \Re(\eig(H)) > 0$ by \eqref{e:exist_harmon_representation_conditions_eigens}, the first bound in (\ref{e:tau(x)_vs_|x|_small_|x|}) implies that $\| c^H \| \to 0$, as $c \to 0$.
Then, $B_{E,H}(0) = B_{E,H}(c^E0) \stackrel{d}=c^H B_{E,H}(0) \rightarrow 0$, as $c \rightarrow 0^+$, and hence $B_{E,H}(0) = 0$ a.s. This implies that $Y = 0$ a.s. in the representation \eqref{e:yaglom_s_representation}.

We now establish that the measure $F$ satisfies the homogeneity relation \eqref{e:int-rep-spectral-F-homog-meas}. By operator self-similarity,
$$
\{c^H B_{E,H}(t)\}_{t \in \bbR^m}  \stackrel{d}= \{B_{E,H}(c^Et)\}_{t \in \bbR^m} \stackrel{d}= \Big\{ \int_{\bbR^m} (e^{i \langle x, c^E t\rangle}-1) \widetilde{B}_F(dx) + X \hspace{1mm}c^Et\Big\}_{t \in \bbR^m}
$$
$$
\stackrel{d}= \Big\{ \int_{\bbR^m} (e^{i \langle y, t\rangle}-1) \widetilde{B}_{F_E}(dy) + X\hspace{1mm}c^Et\Big\}_{t \in \bbR^m},
$$
where we made the change of variables $c^{E^*}x = y$ and the control measure $F_E(\cdot)$ is defined by $F_E(dy) = F(d(c^{-E^*}y))$. However, the measure $F$, the matrix $\Sigma_X$ and the vector $\mu$ are uniquely determined by the field $B_{E,H}$ (Yaglom \cite{yaglom:1957}, Theorem $7'$). Therefore, the relation \eqref{e:int-rep-spectral-F-homog-meas} holds, as well as
\begin{equation}\label{e:scaling_of_the_measure_and_linear_term}
\{X \hspace{1mm}c^{E} t \}_{t \in \bbR^m}\stackrel{{\mathcal L}}= \{c^H X\hspace{1mm} t\}_{t \in \bbR^m}.
\end{equation}
We now show that $X = 0$ a.s. By contradiction, if $X \neq 0$, then there exists an index $i = 1 ,\hdots,n$ such that the corresponding row random vector $X_{i, \bullet}$ is not identically zero. Since the latter is Gaussian, we can write $\Sigma_{i} := EX^*_{i, \bullet}X_{i, \bullet}$, where $\Sigma_{i} \neq 0$. Let $\eta_* = \min \Re(\textnormal{eig}(E))$, and by (\ref{e:exist_harmon_representation_conditions_eigens}), let $\varepsilon > 0 $ be such that $\eta_* - \varepsilon > \max \Re(\textnormal{eig}(H))$. Consider the random (column) vector $c^{E^* - (\eta_* - \varepsilon)} X^*_{i, \bullet}$. Its variance is
$$
M(m,\bbR) \ni \textnormal{Var}(c^{E^* - (\eta_* - \varepsilon)} X^*_{i, \bullet} ) = E(c^{E^* - (\eta_* - \varepsilon)} \hspace{1mm}X^*_{i, \bullet} \hspace{1mm} X_{i, \bullet} \hspace{1mm}c^{E - (\eta_* - \varepsilon)} ) \stackrel{d}= c^{E^* - (\eta_* - \varepsilon)} \Sigma_{i} c^{E - (\eta_* - \varepsilon)}.
$$
Now take a sequence $c_k \rightarrow \infty$ as $k \rightarrow \infty$. Since $\Sigma_{i} \neq 0$, then there exists an associated sequence $t_k \in S^{m-1}$ such that
\begin{equation}\label{e:quadratic_form_-->_infty}
t^*_k c^{E^* - (\eta_* - \varepsilon)}_k \Sigma_{i} c^{E - (\eta_* - \varepsilon)}_k t_k \rightarrow \infty.
\end{equation}
We obtain a Gaussian random sequence $\{X_{i, \bullet} \hspace{1mm}c^{E - (\eta_* - \varepsilon)}_{k} t_{k} \}_{k \in \bbN}\in \bbR$
whose second moment satisfies the relation \eqref{e:quadratic_form_-->_infty}. Such random sequence corresponds to one of the entries on the left-hand side of the scaling expression
\begin{equation}\label{e:scaling_relation_linear_term_after_subtracting_eta}
\bbR^n \ni  X c^{E - (\eta_* - \varepsilon)}_{k} t_{k} \stackrel{d}= c^{H - (\eta_* - \varepsilon)}_{k}  Xt_{k}.
\end{equation}
There is a corresponding scalar-valued entry on the right-hand side of \eqref{e:scaling_relation_linear_term_after_subtracting_eta}. However, since $t_k$ is a unit vector, $c^{H - (\eta_* - \varepsilon)}_{k}  Xt_{k} \stackrel{d}\rightarrow 0$, $k \rightarrow \infty$.
This is a contradiction, since $X$ is Gaussian. Therefore, $X = 0$ a.s., whence representation \eqref{e:int-rep-spectral-F} follows.

We now show \eqref{e:int-rep-spectral-f-1}. Since $X$ is Gaussian and real, we can write $\widetilde{Y}(dx) = \widehat{a}(x) \widetilde{B}(dx)$, where $\widehat{a}(x) = \widehat{f}^{1/2}(x)$. Again by operator self-similarity, second moments and a change-of-variables, we arrive at the relation
$c^H\widehat{a}(x)\widehat{a}(x)^*c^{H^*} = c^{-\textnormal{tr}(E^*)I}\widehat{a}(c^{-E^*}x)\widehat{a}(c^{-E^*}x)^* $ $dx\textnormal{-a.e.}$, $c > 0$. Thus,
\begin{equation}\label{e:a(x)a(x)}
f(c^{E^*}x) = c^{-H_E}f(x) c^{-H^*_E} \quad dx\textnormal{-a.e.}
\end{equation}
By Lemma \ref{l:multidim-mfbm-homog-f-ae}, there exists an $(E,H)$--homogeneous function $\psi$ such that $\psi(x)= f(x)$ $dx$-a.e. Therefore, $\psi$ is positive semi-definite a.e., and the conclusion follows. $\Box$\\
\end{proof}

\begin{remark}
The inequality assumption \eqref{e:exist_harmon_representation_conditions_eigens} generalizes its counterpart for univariate self-similar stochastic processes, namely, $H < E = 1$. However, the inclusion or not of equalities is a more subtle matter. As a consequence of Maejima and Mason \cite{maejima:mason:1994}, Corollary 2.1, the eigenvalues $h_k$, $k = 1, \hdots, n$, of the exponent $H$ of an operator self-similar vector stochastic process must satisfy the constraint $\Re(h_k) \leq E=1$. Nonetheless, one can find matrices $H$ with unit eigenvalues that do not correspond to OFBMs (Didier and Pipiras \cite{didier:pipiras:2011}, Remark 3.2).
\end{remark}

\begin{remark}\label{r:converse_harmonizable_representation}
A converse to the harmonizable representation \eqref{e:int-rep-spectral-f-1} can also be established, namely, that the right-hand side of (\ref{e:int-rep-spectral-f-1}) is a well-defined OFBF. For this purpose, one must show that
\begin{equation}\label{e:exist_harmon_repres_integrability_condition}
\int_{\bbR^m} |e^{i\langle t,x \rangle}-1|^2 \psi(x) dx =: \int_{\bbR^m} h(x)dx< \infty.
\end{equation}
It suffices to check the behavior of the kernel $h(x)$ in \eqref{e:exist_harmon_repres_integrability_condition} for $x \in S_0$ and as $\| x\| \rightarrow 0$ or $\infty$, where the latter limits can be equivalently expressed as $\tau_E(x) \rightarrow 0$ and $\infty$. This can be done conveniently by means of an entrywise application of the change-of-variables formula \eqref{e:polar_coordinates}. Since the latter formula assumes integrability of the original expression, one can build a truncation argument based on $h_A(x) = 1_{\{1/A \leq \tau_E(x) \leq A\}} h(x)$, $A > 0$, and the dominated convergence theorem. It thus suffices to show that the condition
\begin{equation}\label{e:int_polar_domain<infty}
\int^{\infty}_{0} \int_{S_0} |e^{i \langle t, r^{E^*}\theta\rangle}-1|^2 \|r^{-H_E}\|^2 \sup_{\theta \in S_0} \|\psi(\theta) \| r^{\textnormal{tr}(E^*)-1} \sigma(d \theta) dr < \infty,
\end{equation}
expressed in polar coordinates, holds. Indeed, \eqref{e:int_polar_domain<infty} is valid if the assumption \eqref{e:exist_harmon_representation_conditions_eigens} is in place and if $\psi$ is bounded on the sphere $S_0$. This can be established by a multivariate extension of the calculations carried out in Bierm\'{e} et al.\ \cite{bierme:meerschaert:scheffler:2007}, Theorem 4.1.
\end{remark}

\begin{remark}
In contrast with the case where $m=1$ (see Didier and Pipiras \cite{didier:pipiras:2011}, Theorem 3.1), in general one cannot show that the control measure $F$ in Theorem \ref{t:ofbf_harmonizable_repres} is absolutely continuous. A counter-example can be built for $m=2$ and $n=1$ based on a singular measure on $\bbR^2$. Let $m(dx) = 1_{\{x_1 = x_2\}} dx_1$, i.e., the measure of a set is the Lebesgue measure of its intersection with the geometric locus $\{x \in \bbR^2: x_1 = x_2\}$. Define the random field for $t \in \bbR^2$ by
$$
X(t) = \int_{\bbR^2}(e^{i \langle t,x\rangle}-1)\norm{x}^{-(d+1)}\widetilde{Y}(dx),  \quad d
\in (-1/2,1/2),
$$
where $\widetilde{Y}(dx)$ is a Hermitian Gaussian random measure with control measure $E|\widetilde{Y}(dx)|^2 = m(dx)$. Then,
$$
E(X(c^I s)X(c^I t)) = \int_{\bbR^2}(e^{i\langle
cs,x\rangle}-1)(e^{-i\langle ct,x\rangle}-1)\norm{x}^{-2(d+1)} m(dx)
$$
$$
= \int_{\bbR}(e^{i cx(s_1+s_2)}-1)(e^{-icx(t_1+t_2)}-1)2^{-(d+1)}(x^2)^{-(d+1)}dx
$$
$$
=\int_{\bbR}(e^{i y(s_1+s_2)}-1)(e^{-iy(t_1+t_2)}-1)2^{-(d+1)}\Big(\frac{y^2}{c^2}\Big)^{-(d+1)}d\Big(\frac{y}{c}\Big)= c^{2H}E(X(s)X(t)),
$$
where $H = d + 1/2$. The process is well-defined and proper by taking $c = 1$ and $s=t$, noting that the resulting integral is finite and greater than zero. Thus, $X$ is an OFBF with exponents $E=I$ and $H = d + 1/2$.
\end{remark}

\subsection{Moving average representations}

The next theorem, Theorem \ref{t:int-rep-time}, draws upon the harmonizable representation of OFBFs established in Theorem \ref{t:ofbf_harmonizable_repres} to construct a corresponding moving average representation. The spectral filter $g$ in \eqref{e:int-rep-spectral-f-1} is used to generate the real-valued filter $\varphi$ that enters into the latter. Indeed, let
$$
\widehat{f}_t (x) = (e^{i \langle x, t \rangle} -1) g(x) \in L^2(\bbR^m)
$$
be the kernel function appearing in the harmonizable representation of the OFBF $B_{E, H}$. Then, by using Parseval's identity, $B_{E,H}$ can also be represented as
$$ \left\{ B_{E, H} (t)\right\}_{t \in \bbR^m} \ind \left\{ \int_{\bbR^m} f_t(u) B(du) \right\}_{t \in \bbR^m},$$
where
\begin{equation} \label{e:ft(u)} f_t(u) = \frac{1}{(2\pi)^m} \int_{\bbR^m} e^{-i \langle x, u \rangle} \widehat{f}_t(x) dx \quad \left( \in L^2(\bbR^m) \right)
\end{equation}
$$ = \frac{1}{(2 \pi)^m} \int_{\bbR^m} e^{-i \langle x, u \rangle} ( e^{i \langle x, t \rangle} -1 ) g(x) dx $$
\begin{equation} \label{e:ft(u)-1}
=  \frac{1}{(2 \pi)^m} \int_{\bbR^m} \left( \cos(\langle x, t-u \rangle) - \cos \left( \langle x, -u \rangle \right) \right) g(x) dx.
\end{equation}
The bulk of the proof of Theorem \ref{t:int-rep-time} amounts to expressing the integral (Fourier transform) \eqref{e:ft(u)-1} as $\varphi(t-u) - \varphi(-u)$, for suitable $\varphi$ (see expressions \eqref{e:multidim-mfbm-rep-time-proof-1} and \eqref{e:multidim-mfbm-rep-time-proof-4} in the proof). Note that the construction of this connection between the Fourier and parameter domains demands additional assumptions on the spectral filter function $g$ and on the exponents $E$ and $H$.

\begin{theorem}\label{t:int-rep-time}
Let $B_{E,H} = \{B_{E,H}(t)\}_{t \in \bbR^m}$ be an OFBF with exponents $(E,H)$. Assume that the following conditions are in place:
\begin{itemize}
\item[$(c.1)$] the spectral filter function $g$ in \eqref{e:int-rep-spectral-f-1} is differentiable in the sense that
\begin{equation}\label{e:int-rep-time-cond-1}
\frac{\partial^m}{\partial x_1\ldots \partial x_m} g(x) = \widetilde g(x) \quad dx\mbox{-a.e.},
\end{equation}
where $x = (x_1, \ldots, x_m)$ and the differentiation above is taken entry-wise;
\item[$(c.2)$] $\widetilde g$ is bounded on $S_0$;
\item[$(c.3)$] the matrix exponents $E$, $H$ are diagonalizable with real eigenvalues and real eigenvectors; and
\item[$(c.4)$]
\begin{equation}\label{e:int-rep-time-cond-2}
\max \eig(H) + \tr(E)/2 < m \min \eig(E).
\end{equation}
\end{itemize}
Then, there is a function $\varphi:\bbR^m \to \bbR^n$, depending on $E$ and $H$, which satisfies
\begin{itemize}
\item [(i)] $\varphi(t-\cdot) - \varphi(-\cdot) \in L^2(\bbR^m)$, $t \in \bbR^m$;
\item [(ii)] $\varphi$ is $(E,H - \tr(E) I/2)$--left-homogeneous,
\end{itemize}
and such that
\begin{equation}\label{e:int-rep-time}
\{B_{E,H}(t)\}_{t\in \bbR^m} \stackrel{{\mathcal L}}{=} \Big\{ \int_{\bbR^m} (\varphi(t-u) - \varphi(-u)) B(du) \Big\}_{t\in\bbR^m},
\end{equation}
where $B(du)$ is an $\bbR^n$--valued Gaussian random measure on $\bbR^m$ with Lebesgue control measure.
\end{theorem}
\begin{proof}
It suffices to consider the case where the exponents $E \in M(m)$, $H \in M(n)$ are diagonal. Indeed, in this case we can write $E = W E_0 W^{-1}$ and $H = V H_0 V^{-1}$, where $E_0 \in M(m)$ and $H_0 \in M(n)$ are diagonal matrices, and $W \in M(m)$, $V \in M(n)$. The newly defined field $\widetilde{B}_{E_0, H_0}(t) = V^{-1} B_{E,H}(Wt)$ is then an OFBF with diagonal exponents $E_0$ and $H_0$, and can be represented as
$$
\left\{ \widetilde{B}_{E_0, H_0} (t) \right\}_{t \in \bbR^m} \ind \left\{ \int_{\bbR^m} (\varphi_0(t-u) - \varphi_0(-u)) B(du) \right\}_{t \in \bbR^m},
$$
where $\varphi_0$ is $(E_0, H_0 - \textnormal{tr}(E_0)I/2)$--left-homogeneous. Therefore,
\begin{eqnarray}
\left\{ B_{E,H}(t) \right\}_{t \in \bbR^m} & \ind & \left\{ \int_{\bbR^m} (V \varphi_0(W^{-1}t - u) - V \varphi_0(-u) ) B(du) \right\}_{t \in \bbR^m} \\
& \ind & \left\{ \int_{\bbR^m} (V \varphi_0 (W^{-1}(t - v)) - V \varphi_0(- W^{-1}v)) \left|{\rm det}(W) \right|^{-1/2} B(dv) \right\}_{t \in \bbR^m} \\
& \ind & \left\{ \int_{\bbR^m}  \left(\varphi(t-v) - \varphi(-v) \right) B(dv)  \right\}_{t \in \bbR^m},
\end{eqnarray}
after the change of variables $u = W^{-1}v$ (so that $du = |{\rm det}(W)|^{-1} dv$), where
$$ \varphi(v) = |{\rm det}(W)|^{-1/2} V \varphi_0 (W^{-1} v).
$$
By using the fact that $\tr(E_0) = \tr(E)$, we obtain
$$ \varphi(c^E u) = |{\rm det}(W)|^{-1/2} V \varphi_0 (W^{-1} c^E v) = |{\rm det}(W)|^{-1/2} V \varphi_0 (W^{-1} W c^{E_0} W^{-1} v) $$
$$ = |{\rm det}(W)|^{-1/2} V \varphi_0 (c^{E_0} W^{-1} v) = |{\rm det}(W)|^{-1/2} V c^{H_0 - \tr(E_0)I /2 } \varphi_0 (W^{-1} v) $$
$$ = |{\rm det}(W)|^{-1/2} c^{H - \tr(E)I /2 } V \varphi_0 (W^{-1} v) = c^{H -\tr(E)I/2} \varphi(v).
$$
Therefore, the function $\psi$ is $(E,H - \tr(E) I/2)$--left-homogeneous.

Thus, suppose without loss of generality that $E \in M(m)$ and $H \in M(n)$ are diagonal. Let $f_t(u)$ be an in \eqref{e:ft(u)-1}. We want to express such function as $\varphi(t-u)-\varphi(-u)$, where $\varphi$ is $(E,H - \tr(E)I/2)$--left-homogeneous (some of the technical difficulties involved are discussed in Remark \ref{re:tech diff} below). Since $H \in M(n)$ is diagonal, it suffices to consider the case $n=1$, though the argument also applies when $n\ge 1$ by considering the pertinent entry-wise expressions.

After a change-of-variables in each orthant and in view of the fact that $g$ is symmetric, we can write
\begin{equation}\label{e:multidim-mfbm-rep-time-proof-1}
    (2\pi)^m f_t(u) = \sum_\sigma \int_{[0,\infty)^m} ( \cos(\langle x,(t-u)_\sigma\rangle) - \cos(\langle x,(-u)_\sigma \rangle) ) g(x) dx =: \sum_\sigma f_{\sigma,t}(u),
\end{equation}
where the sum is over all $\sigma = (\sigma_1,\ldots,\sigma_m)$ with $\sigma_j\in\{-1,1\}$,  and $(v)_\sigma = (\sigma_1v_1,\ldots,\sigma_m v_m)$. It is then enough to show that
\begin{equation}\label{e:multidim-mfbm-rep-time-proof-4}
  f_{\sigma,t}(u) = \varphi_\sigma(t-u) - \varphi_\sigma(-u)
\end{equation}
for some $(E,H-\tr(E)/2)$--left-homogeneous $\varphi_\sigma$. As with $f_t(u)$ in (\ref{e:ft(u)}), each integral $f_{\sigma,t}(u)$ is defined in the $L^2(\bbR)$ sense, that is, as the inverse Fourier transform of an $L^2(\bbR)$ function.

For notational simplicity, set
$$
q = \textnormal{tr}(E) = \textnormal{tr}(E^*).
$$
Since $g$ is $(E^*,-H-q/2)$--left-homogeneous and differentiable in the sense of (\ref{e:int-rep-time-cond-1}), we have
\begin{equation}\label{e:multidim-mfbm-rep-time-proof-2}
    g(x) = \int_{[x,\infty)^m} \widetilde g(y) dy,
\end{equation}
where $\widetilde g$ is $(E^*,-H-3q/2)$--left-homogeneous. Then,
$$
f_{\sigma,t}(u) = \int_{[0,\infty)^m} ( \cos(\langle x,(t-u)_\sigma\rangle) - \cos(\langle x,(-u)_\sigma \rangle) ) \Big(  \int_{[x,+\infty)^m} \widetilde g(y) dy \Big) dx
$$
\begin{equation}\label{e:multidim-mfbm-rep-time-proof-3}
  = \int_{[0,\infty)^m} dy \widetilde g(y) \Big( \int_{[0,y]^m} \cos(\langle x,(t-u)_\sigma\rangle) dx
  - \int_{[0,y]^m} \cos(\langle x,(-u)_\sigma\rangle) dx  \Big),
\end{equation}
where the change of the order of integration will be justified at the end of the proof. Observe now that
$$
\int_{[0,y]^m} \cos(\langle x,v\rangle) dx = \Re \int_{[0,y]^m} e^{i \langle x,v\rangle} dx =
\frac{1}{v_1 \ldots v_m} \Re \Big(i^{-m} (e^{iy_1v_1}-1)\ldots (e^{iy_mv_m}-1) \Big).
$$
Then, by (\ref{e:multidim-mfbm-rep-time-proof-3}), the relation (\ref{e:multidim-mfbm-rep-time-proof-4}) holds with
$$
\varphi_{\sigma}(v) = \int_{[0,\infty)^m} \frac{\widetilde g(y)}{(v_1 \sigma_1) \ldots (v_m \sigma_m)} \Re \Big(i^{-m} (e^{iy_1(v_1 \sigma_1)}-1) \ldots (e^{iy_m (v_m \sigma_m)}-1) \Big)dy
$$
\begin{equation}\label{e:multidim-mfbm-rep-time-proof-5}
  =: \int_{[0,\infty)^m} \frac{\widetilde g(y)}{(v_1 \sigma_1)\ldots (v_m \sigma_m)} k_{\sigma,v}(y) dy =: \frac{\varphi_{0,\sigma}(v)}{(v_1 \sigma_1) \ldots (v_m \sigma_m)}.
\end{equation}
Since $E$ is diagonal, the function $\varphi_{\sigma}(v)$ can be checked to be $(E,H-qI/2)$--left-homogeneous, as required. It is then enough to show that this function, or equivalently, the function $\varphi_{0,\sigma}(v)$, is well-defined as a Lebesgue integral.

By the formula (\ref{e:polar_coordinates}) and a truncation argument (see also Remark \ref{r:converse_harmonizable_representation}),
$$
|\varphi_{0,\sigma}(v)| \leq \int_{[0,\infty)^m} |\widetilde g(y)| |k_{\sigma,v}(y)| dy
$$
\begin{equation}\label{e:multidim-mfbm-rep-time-proof-6}
= \int_0^\infty dr\, r^{-H -3q/2} r^{q-1}
\Big( \int_{S_0} |\widetilde g(\theta)| |k_{\sigma,v}( r^{ E^* } \theta)| \sigma(d\theta)\Big),
\end{equation}
where we used the fact that $\widetilde g(r^{E^* } \theta) = r^{-H-3q/2} \widetilde g(\theta)$ by left-homogeneity. Note that the integral over $S_0$ is bounded by a constant. Therefore, around $r= \infty$, the integrand behaves like the power law $r^{- H - q/2 -1}$, and is integrable since
$$
(-H-q/2-1) + 1 = -H - q/2 < 0.
$$
As for the behavior of the integrand around $r=0$, observe that, for $y_1\geq 0,\ldots,y_m\geq 0$, $|k_{\sigma,v}(y)| \leq C y_1\ldots y_m \leq C (y_1 + \ldots + y_m)^m \leq C' \|y\|^m$. Then, the integrand is bounded by
$$
r^{-H-q/2-1} \|r^{E^*} \|^m \le C r^{-H-q/2-1} r^{m e_{\min}},
$$
where $e_{\min} = \min \eig(E)$.
The integrability around $r=0$ follows in view of the assumption \eqref{e:int-rep-time-cond-2}, since
$$
-H - q/2 - 1 + m e_{\min} + 1 = m e_{\min} - H - q/2 > 0.
$$

Finally, we justify the change of the order of integration in (\ref{e:multidim-mfbm-rep-time-proof-3}). Since $f_{\sigma, t}(u)$ is defined in the $L^2(\bbR)$ sense, we have
\begin{equation}\label{e:f_sigma_L2}
f_{\sigma, t}(u) = \lim_{\ell \to \infty} (L^2(\bbR)) \int_{[0, \ell)^m}
\left( \cos( \langle x, (t-u)_\sigma \rangle ) - \cos (\langle x, (-u)_\sigma \rangle ) \right) g(x) dx.
\end{equation}
The idea now is to apply Fubini's theorem in the truncated integral (\ref{e:f_sigma_L2}), and show that the resulting expression converges to the right-hand side of (\ref{e:multidim-mfbm-rep-time-proof-3}) as $\ell \to \infty$. Let $L = \left(\ell, \ell, \ldots, \ell \right) \in \bbR^m$. Then,
$$ \int_{[0, \ell)^m} ( \cos(\langle x, (t-u)_\sigma \rangle ) - \cos (\langle x, (-u)_\sigma \rangle ) ) \left( \int_{[x, \infty)^m} \widetilde{g}(y) dy \right) dx $$
$$ = \int_{[0, \ell)^m} ( \cos(\langle x, (t-u)_\sigma \rangle ) - \cos (\langle x, (-u)_\sigma \rangle ) )  \left( \int_{[x, \ell)^m} \widetilde{g}(y) dy + g(L) \right) dx $$
$$ = \int_{[0, \ell)^m} dy \widetilde{g}(y) \left( \int_{[0, y]^m}  ( \cos(\langle x, (t-u)_\sigma \rangle ) - \cos (\langle x, (-u)_\sigma \rangle ) ) dx \right) $$
$$ + g(L) \int_{[0, \ell)^m}  ( \cos(\langle x, (t-u)_\sigma \rangle ) - \cos (\langle x, (-u)_\sigma \rangle ) )  dx =: I_1 (L) + g(L) I_2(L).$$
The first term $I_1(L)$ converges to the right-hand side of (\ref{e:multidim-mfbm-rep-time-proof-3}) as $L \to \infty$. Indeed, as shown following (\ref{e:multidim-mfbm-rep-time-proof-3}), the integrand is in $L^1([0, \infty)^m)$. The second term $g(L)I_2(L)$ converges to zero since $g(L) \to 0$ and $|I_2(L)|$ is bounded. $\Box$\\
\end{proof}

\begin{remark}\label{re:tech diff} It is illuminating to revisit the technical difficulties involved in the proof of Theorem \ref{t:int-rep-time} by considering the univariate, uniparameter context $m=n=1$. In the latter case, one can pick the natural homogeneous specification $g(x) = c |x|^{-H - 1/2}$. Under such choice, the integral (\ref{e:ft(u)}) can be recast as
$$ c \int_\bbR (\cos(x(t-u)) - \cos(x(-u))) |x|^{-H-1/2} dx = \varphi(t-u) - \varphi(-u), $$
where
\begin{equation}\label{e:psi(v)}
\varphi(v) = \left\{ \begin{array}{l}
\mbox{$\displaystyle c \int_\bbR \cos(xv) |x|^{-H-1/2} dx, \quad \mbox{if~} H < 1/2;$} \vspace{ 2 mm}\\
\mbox{$\displaystyle c \int_\bbR (\cos(xv) -1 ) |x|^{-H-1/2} dx, \quad \mbox{if~} H > 1/2.$} \\
\end{array} \right.
\end{equation}
The expression \eqref{e:psi(v)} is interpreted as involving an improper Riemann integral when $H < 1/2$ and a Lebesgue integral when $H > 1/2$. Neither interpretation carries over to the multiparameter case $m \ge 2$. On one hand, when $m \ge 2$, left-homogeneous functions are no longer simple power functions as above. On the other hand, when $n=1$, the integral
\begin{equation}\label{e:int_(cos-1)g(x)dx}
\int_{\bbR^m} (\cos(\langle x, v \rangle) -1 ) g(x) dx,
\end{equation}
where $g$ is an $(E^*, -H_E)$--(left-)homogenous function, is definable as a Lebesgue integral when $2 H > q$. However, the condition \eqref{e:exist_harmon_representation_conditions_eigens} for the construction of a harmonizable representation yields the parallel constraint $q > m H$. As a consequence, \eqref{e:int_(cos-1)g(x)dx} can only play the role of the fractional filter $\varphi(v)$ in the uniparameter context, namely, when $m = 1$.
\end{remark}

\begin{remark} Sufficient conditions for the integral (\ref{e:int-rep-time}) to be well-defined can be given using Theorem 3.1 in Bierm\'{e} et al.\ \cite{bierme:meerschaert:scheffler:2007}. Suppose for simplicity that $H$ is diagonal with real eigenvalues and real eigenvectors (more generally, the argument below can be adapted to diagonalizable $H$). Assume also that the entry-wise functions $\varphi_{ij}$ in the matrix $\Phi = (\varphi_{ij})_{i,j = 1, \ldots, n}$ are non-negative, and write $\varphi_{ij}(x) =: \xi_{ij}(x)^{h_i - q/2}$, where $q = \tr(E)$. Since each $\varphi_{ij}$ is $(E, h_i - q/2)$--(left-)homogeneous, the function $\xi_{ij}$ is $(E,1)$--homogeneous. Following Definition 2.7 of Bierm\'{e} et al.\ \cite{bierme:meerschaert:scheffler:2007}, suppose the function $\xi_{ij}$ is $(\beta_{ij}, E)$--admissible, that is, for all $x \neq 0$ and $0 < A < B$, there is a constant $C > 0$ such that, for $A \le \| y \| \le B$,
$$ \tau_E(x) \le 1 \Rightarrow |\xi_{ij} (x + y) - \xi_{ij}(y) | \le C \tau_E(x)^{\beta_{ij}}. $$
Then, by Theorem 3.1 in Bierm\'{e} et al.\ \cite{bierme:meerschaert:scheffler:2007}, the integral (\ref{e:int-rep-time}) is well-defined if $0 < h_i < \beta_{ij}$, $i,j=1, \ldots, m$.
\end{remark}

\bibliography{ofbf}

\small

\bigskip

\flushleft
\begin{tabular}{lp{1.9 in}l}
Changryong Baek & & Gustavo Didier\\
Dept.\ of Statistics & & Mathematics Department \\
Sungkyunkwan University & & Tulane University\\
25-2, Sungkyunkwan-ro, Jongno-gu & & 6823 St.\ Charles Avenue   \\
Seoul, 110-745, Korea & & New Orleans, LA 70118, USA \\
{\it crbaek@skku.edu} & &{\it gdidier@tulane.edu} \\
\end{tabular}

\bigskip
\noindent \begin{tabular}{lcl}
Vladas Pipiras & \hspace{3.5cm} &  \\
Dept.\ of Statistics and Operations Research & & \\
UNC at Chapel Hill & &  \\
CB\#3260, Hanes Hall  & & \\
Chapel Hill, NC 27599, USA & &  \\
{\it pipiras@email.unc.edu}& &  \\
\end{tabular}\\

\end{document}